%% file: main.tex
\documentclass[11pt]{article}
\usepackage[utf8]{inputenc}
\usepackage[
papersize={5.5in, 8.5in},
left=0.35in,
right=0.35in,
top=0.35in,
bottom=0.5in]{geometry}
\usepackage{graphicx}
\usepackage{amsthm}
\usepackage{rotating}
\usepackage{amsfonts,amssymb,amsmath}
\usepackage[usenames]{color}
\usepackage[T1]{fontenc}

\usepackage{array}
\usepackage{float}
\usepackage{booktabs}
\usepackage[dvipsnames,usenames,table,xcdraw]{xcolor}

\graphicspath{ {images/} }
\newtheorem*{theorem}{Theorem}
\title{\bf 
On the plane and its coloring
}
\author{\bf 
\textcolor[rgb]{0.8,0,1}{Jaan Parts} \\
} 
\date{\normalsize \textcolor[rgb]{0.8,0,1}{Kazan, Russia, jaan\_parts@.mail.ru}}

\begin{document}

\maketitle

\pagestyle{empty}
\thispagestyle{empty}

\begin{abstract}
We provide a human-verifiable proof that, in a certain sense, the chromatic number of the plane is exactly 7.
\end{abstract}

\input{introduction}

\input{poly}

\input{proof}

\input{add}

\input{references}
\end{document}

%% file: introduction.tex
\section{Introduction}

It is generally accepted that the rainbow contains seven colors. Although in the traditions of different nations, the sets of these colors are different. According to the legend, Isaac Newton identified five primary colors: red, yellow, green, blue and violet. Later, he added two more: orange and indigo, as he believed that the number of primary colors corresponds to the number of primary notes in an octave. There is no indigo in Russian, but there is a word for light blue. But according to the number of primary colors, there seems to be agreement.
Following Newton, we want to see a similar color harmony for the plane.

The original title 
was "The chromatic number of the plane is 7 $-$ a human-verifiable proof". But my dear editors dissuaded me from using it.

On the one hand, there is no mistake here: we will indeed be talking about the \textit{exact} 
number of colors.
Recall that the \textit{chromatic number of the plane} (denoted $\chi$) is the minimum number of sets into which the Euclidean plane can be partitioned such that the distance between any two points taken from the same set is not equal to some \textit{forbidden distance} $f$, which is usually taken $f=1$. 
It is customary to assign a \textit{color} to each such set (along with an ordinal number) and talk about a \textit{coloring} of the plane. Any coloring with a forbidden distance is called \textit{proper}.

It is known that 7 colors are enough for a proper coloring of the plane. This is easy to demonstrate by \textit{tiling} a plane in the form of a honeycomb using regular hexagons. Each such hexagon takes one of seven colors and has a diameter slightly less than one. This construction, whose authorship is attributed to John Isbell \cite{soi}, gives an estimate $\chi\le 7$ at once for the entire \textit{interval} of forbidden distances $1\le f\le\sqrt7/2\approx 1.322876$.

To obtain a \textit{lower bound} of $\chi$, a (finite) \textit{graph} is usually constructed with the corresponding value of $\chi$, but other considerations are possible. By the way, we are not the first to announce the proof of the exact value of $\chi$. For example, Kai-Rui Wang proved \cite{wang} that $\chi=4$. True, this was before Aubrey de Grey showed \cite{grey} that $\chi\ge 5$. Alexander Soifer conjectured \cite{soi} that $\chi=7$, which is the value we will focus on.

On the other hand, since this article was being prepared for the April issue of the journal, we considered it possible to remain silent (temporarily) that we will be talking about the chromatic number of a plane with an interval of forbidden distances $1\le f\le d$ \:\: $(d>1)$,
moreover, about a specific set of such intervals. 
We hasten to assure the reader that everything below will be honest, without deception.

Some forbidden distance intervals 
have one surprising property: there are values of $d$ for which $\chi$ is known exactly. This fact was discovered by Geoffrey Exoo \cite{exoo} who constructed a 203-vertex 7-chromatic graph with 
$d=\sqrt{43/25}\approx 1.311488<\sqrt7/2$ . If we recall the whole history of colorings \cite{soi}, in most generalizations of the plane chromatic number problem, the gap between the known lower and upper bounds of $\chi$ only gets larger, so the existence of such islands of exact knowledge is like a miracle.

Recently, Joanna Chybowska-Sokół, Konstanty Junosza-Szaniawski, and Krzysztof W\k{e}sek considered a graph of a different construction \cite{wes}, which (among other things) allowed them to get $d\approx 1.285987$.

Here we go a little further and, for $d=2\sin(2\pi/9)\approx 1.285575$, 
offer a simple constructive proof that can be verified without using a computer.

%% file: poly.tex
\section{Polychromatic vertices}

In some cases, instead of the forbidden interval of the asymmetric form $1\le f\le d$, it is more convenient to consider the symmetric interval $1-\varepsilon \le f\le 1+\varepsilon$. 
These forms are interchangeable, 
$d=(1+\varepsilon)/(1-\varepsilon)$.

The transition from one forbidden distance ($\varepsilon=0$) to a non-zero interval of forbidden distances ($\varepsilon>0$) significantly affects coloring. For example, two colors are no longer enough to color a straight line $\mathbb{R}^1$.

To show that the chromatic number of the plane $\mathbb{R}^2$ with $\varepsilon>0$ is not less than 6, only three vertices of a regular unit triangle are sufficient. Indeed, for any fixed proper coloring of the plane, there is a point \textit{A} whose $\varepsilon$-neighbourhood contains three colors 1, 2, 3. On a circle of unit radius centered at point \textit{A}, there is a point \textit{B} whose $\varepsilon$-neighborhood contains two other colors 4, 5. If we place two vertices of the triangle at points \textit{A} and \textit{B} (let's call them \textit{tri-} and \textit{bi-chromatic}, respectively), the remaining vertex \textit{C} will have color 6.
It remains to prove the existence of tri- and bi-chromatic vertices for any proper coloring of the plane, which is not very difficult (see lemma 3.1 in \cite{wes}).

Note that this construction can be generalized to an Euclidean space $\mathbb{R}^n$ of arbitrary dimension 
$n\in\mathbb{Z}_{\ge 0}$
: for any proper coloring, there is such an $(n+1)$-vertex regular unit simplex, whose vertices will occupy sequentially $\{n+1,\; n,\; n-1,\, \dots, 1\}$ colors. In other words, the chromatic number of $\mathbb{R}^n$ with $\varepsilon>0$ is at least ${n+2 \choose 2}=(n+1)(n+2)/2$.

%% file: proof.tex
\section{Main construction}

We will rely on the construction proposed in \cite{wes}. Namely, a 4-chromatic graph is constructed, all vertices of which are located in the annulus between circles with radii 1 and $d$ and are connected by edges of length from 1 to $d$ for some $d\ge 1$. An additional central vertex is connected by edges to all annulus vertices and is located on a properly colored plane so that its $\varepsilon$-neighborhood contains points of at least three colors. In other words, the central vertex is tri-chromatic, which results in a 7-chromatic graph.

In fact, we took the 2601-vertex graph from \cite{wes} (claim 3.3, construction 1), and reduced it a little bit. One circle of unit radius was enough for us, on which 18 vertices are evenly distributed, including one bi-chromatic vertex. The resulting 19-vertex graph (Fig.~\ref{g19}) is already 
suitable for use in the proof of the

\begin{figure}[!t]
\centering
\includegraphics[scale=0.22]{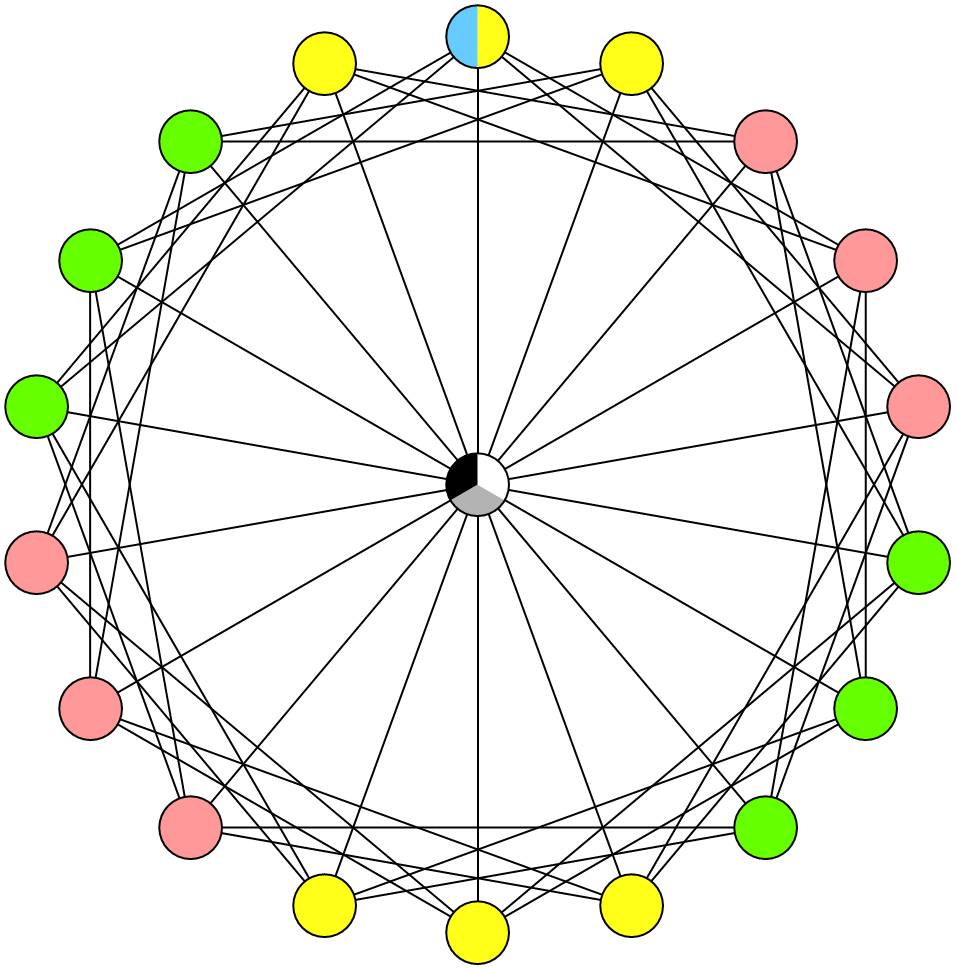}
\smallskip
\caption{A 19-vertex 7-chromatic graph with bi- and tri-chromatic vertices. A variant of the 7-coloring is shown.}
\label{g19}
\end{figure}

\begin{theorem}
{\upshape \cite{wes}, \cite{exoo}} 
The chromatic number of the plane with an interval of forbidden distances\, $[\,1,d\,]$\,
is 7\, for\, $d\in (\,2\sin(2\pi/9), \sqrt7/2\,]$.
\end{theorem}

\begin{proof}
Consider the 19-vertex graph shown in Fig.~\ref{g19}. We temporarily discard the central vertex and renumber the remaining 18 vertices sequentially clockwise from 1 to 18. For convenience, we can arrange them in a row, as shown in Fig.~\ref{proof}. We will identify the vertices and their numbers (indices), taking into account that the indexing is done modulo 18. 

Since the 18-vertex graph contains an odd cycle $C_5$, at least three colors are required (numbered 1, 2, 3 in Fig.~\ref{proof}). Let $m$ denote the number of consecutive (from $i$ to $i+m-1$) vertices of the same color, bounded by two vertices ($i-1$ and $i+m$) of other colors. We will show successively that in any proper 3-coloring a) only pairs ($m=2$) and triplets ($m=3$) are allowed, b) they cannot occur simultaneously, c) a bi-chromatic vertex is not allowed.
Note that each vertex $i$ of the 18-vertex graph has degree 4 and is connected to vertices $i\pm 3$ and $i\pm 4$. Therefore, $m\le 3$, and the coloring rule applies: 1) two consecutive vertices $i$ and $i+1$ of two different colors give the remaining third color to the vertices $i-2$ and $i+3$.

To show $m\neq 1$, it suffices to make sure that the sequences of colors 123 and 121 are not allowed. To do this, we assign these colors to the three consecutive vertices, after which, based on the rule described above, we sequentially color some other vertices, until we come to a contradiction.

Case 123 is shown in Fig.~\ref{proof}\textit{a}. The vertices are colored in the following sequence (arrow means the colors are uniquely defined by specified vertices): let vertices 7, 8, 9 be colors 1, 2, 3;
then $\{7,8\}\Rightarrow \{4,11\}$; $\{ 8,9\}\Rightarrow \{5,12\}$; $\{4,5\}\Rightarrow 1$; $\{11,12\}\Rightarrow 15$; this leads to a conflict: vertices 1 and 15 
connected by an edge take the same color.

Case 121 is shown in Fig.~\ref{proof}\textit{b}: assign colors to $\{7,8,9\}$; then $\{7,8\}\Rightarrow 11$; and whatever color the vertex $10$ takes, we get three different colors of $\{8,9,10\}$ or $\{9,10,11\}$, which brings us to the case 123 considered earlier.

Thus, the colors necessarily occur either in pairs or in triplets; the vertex $i$, both nearest vertices of which $i\pm 1$ have other colors, is forbidden. In other words, the following coloring rules are added: 2) two consecutive vertices $i$ and $i+1$ of different colors give these colors to the vertices $i-1$ and $i+2$ respectively; 3) a pair or triple of vertices of the same color is bounded by vertices of two different colors (because of the edges connecting the vertices $i$ and $i+3$).

Let us show that pairs and triplets cannot alternate in any proper 3-coloring. To do this, consider the initial sequence of colors 331122233 (containing a pair and a triple), and continue coloring using the rules above (Fig.~\ref{proof}\textit{c}): specify colors of $\{5,6,\dots, 13\}$; then $\{6,7\}\Rightarrow 3$; $\{11,12\}\Rightarrow 15$; $\{3,15\}\Rightarrow 18$; $\{7,18\}\Rightarrow 4$; $\{4,5\}\Rightarrow 1$; which leads to a conflict in the coloring of the vertices 1 and 15. 

Thus, only two different proper 3-colorings are possible, containing either six triples or nine pairs. Now it is easy to check that any vertex is connected by edges to vertices of two different colors. Hence, at least four colors are required to color a 18-vertex graph with one bi-chromatic vertex. 
(In any of the two possible 3-colorings, there will be a sequence 1$x$223, where $x$ is 1 or 2, and the first vertex of this sequence cannot be changed to a bi-chromatic one.)

It remains to place in the center one tri-chromatic vertex connected by edges to all other vertices, which will result in a 7-chromatic 19-vertex graph shown in Fig.~\ref{g19}.
To complete the proof, note that this graph has edges of length 1 and $2\sin(2\pi/9)$, the latter gives the lower bound on $d$, while the upper bound $\sqrt7/2$ is given by Isbell's tessellation.

\end{proof}

\begin{figure}[!b]
\centering
\includegraphics[scale=0.275]{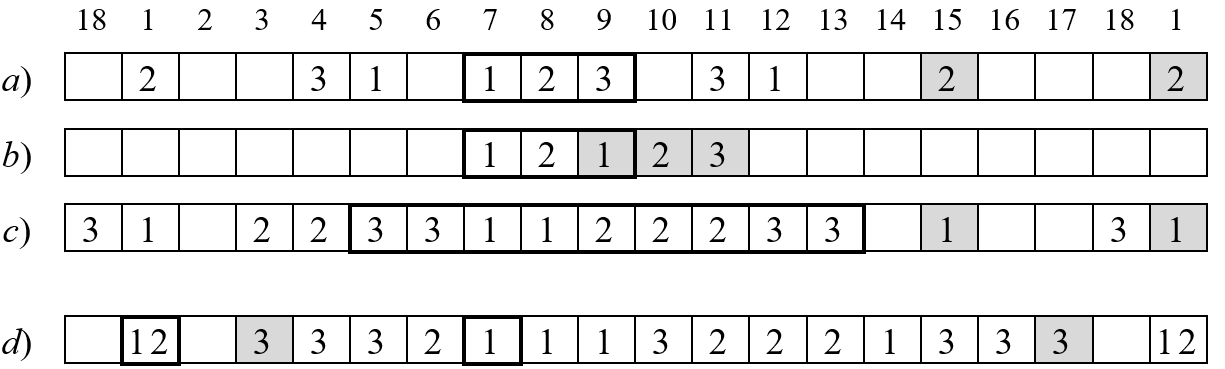}
\smallskip
\caption{To the proof of the Theorem. The vertex numbers are shown at the top, then below are the specific cases of coloring. The initial color assignments are bordered. Conflicts are shown in gray.}
\label{proof}
\end{figure}

%% file: add.tex

\section{Additions
}

Aubrey de Grey suggested starting the coloring from a bi-chromatic vertex, which simplifies the proof (see Fig.~\ref{proof}\textit{d}).
Let vertex 1 of 18-vertex graph be colors 1 and 2. Then, in any proper 3-coloring vertices 4, 5, 15 and 16 must be color 3, so vertices 7, 8, 9, 11, 12 and 13 cannot be color 3. Then, without loss of generality let vertex 7 be color 1. 
Then $\{7,15\}\Rightarrow 11$, $\{5,11\}\Rightarrow 8$, $\{8,15\}\Rightarrow 12$, $\{5,12\}\Rightarrow 9$, $\{9,16\}\Rightarrow 13$, $\{7,13\}\Rightarrow 10$, $\{9,10\}\Rightarrow 6$, $\{10,11\}\Rightarrow 14$, $\{6,7\}\Rightarrow 3$, $\{13,14\}\Rightarrow 17$, so vertices 3 and 17 are both color 3, contradiction.

Who can do better?

However, it all depends on the criteria. The second version is more straightforward, requires less checks, looks more compact. The first version probably allows a better understanding of what is happening (since it gives all colorings without a bi-chromatic vertex, which are already limited) and requires fewer coloring operations (if we don’t count initial color assignments).

It can be seen that in second version of the proof, two vertices are redundant. 
If vertex 1 is bi-chromatic, then one of the following pairs of vertices can be removed: $\{2,18\}$, $\{2,6\}$, $\{8,12\}$, or $\{14,18\}$. In the first version of the proof, four vertices are not involved in the coloring.

Aubrey also suggested using several bi-chromatic vertices instead of one. Thus, one can try to reduce the number of vertices involved in the proof. Recall that a bi-chromatic vertex appears at the boundary of two regions of different colors. Apparently, at least 9 such vertices are required. Indeed, if we take a circle with a radius slightly greater than one, then for its 3-coloring with a forbidden unit distance, at least 7 segments are required. But if there are 7 or 8 segments, then one of them will border on two segments of the same color, and as a result, there will be a pair of points of the same color at a unit distance from each other. Therefore, at least 9 segments are required, which means 9 bi-chromatic vertices.

One can do without a bi-chromatic vertex, for example, by considering a 29-vertex graph shown in Fig.~\ref{g29} with edge length between 1 and $d=2\sin(5\pi/22)\approx 1.309721$. Here all vertices, except for the central one, occupy two concentric circles of radii 1 and $d$.
But the proof with the 19-vertex graph seems simpler, and $d$ turns out to be smaller.

\begin{figure}[!t]
\centering
\includegraphics[scale=0.22]{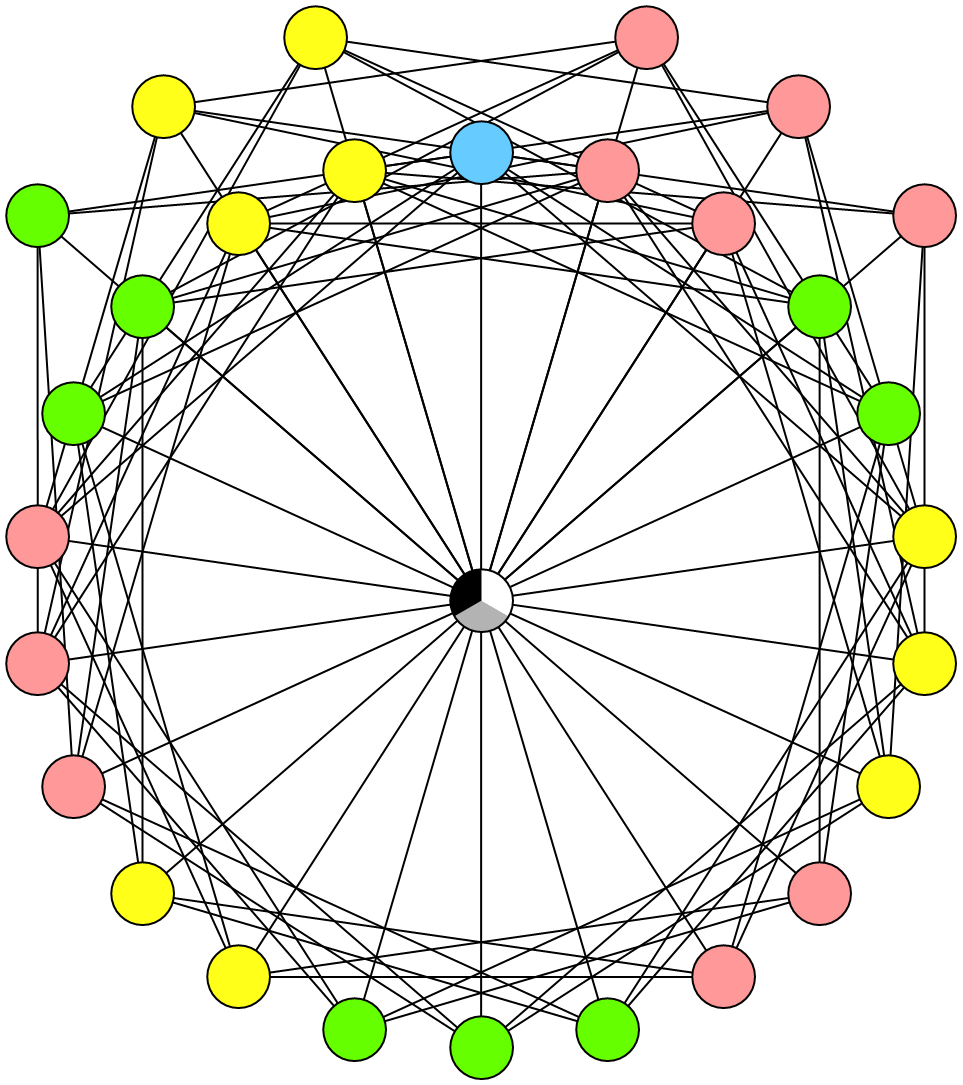}
\smallskip
\caption{A 29-vertex 7-chromatic graph with tri-chromatic vertex.}
\label{g29}
\end{figure}

It would be good to go further with a decrease in $d$, but apparently this cannot be done within the framework of such constructions. The value of $d$ is limited by the existence of a three-color tiling \cite{wes} of an annulus with an outer radius $2\sin(2\pi/9)$.

I thank Alexander Soifer and Aubrey de Grey, who found an error in every second statement. But it's better to check everything anyway.